\documentclass[reqno]{amsart}
\usepackage[utf8]{inputenc}
\usepackage{amsmath}
\usepackage{amssymb}
\usepackage{amsthm}
\usepackage{amsopn}
\usepackage{color}
\usepackage{nicefrac}
\usepackage{enumerate}
\usepackage[hidelinks]{hyperref}
\usepackage[notref,notcite]{}
\usepackage{soul}

\newtheorem{theorem}{Theorem}[section]
\newtheorem{lemma}[theorem]{Lemma}

\newtheorem{proposition}[theorem]{Proposition}
\newtheorem{corollary}[theorem]{Corollary}
\theoremstyle{remark}
\newtheorem{remark}[theorem]{Remark}

\newcommand{\R}{\mathbb{R}}
\newcommand{\RN}{\mathbb{R}^N}
\newcommand{\TN}{\mathbb{T}^N}

\begin{document}

	\title[Rates of convergence for nonlocal HJB equations]{Rates of Convergence in Periodic Homogenization of Nonlocal Hamilton-Jacobi-Bellman Equations}
	
	\author[]{Andrei Rodríguez-Paredes}
	\address{
		Andrei Rodr\'iguez-Paredes:
		Departamento de Matem\'atica y C.C., Universidad de Santiago de Chile,
		Casilla 307, Santiago, CHILE.
		\newline {\tt andrei.rodriguez@usach.cl}
	}

	\author[]{Erwin Topp}
	\address{
		Erwin Topp:
		Departamento de Matem\'atica y C.C., Universidad de Santiago de Chile,
		Casilla 307, Santiago, CHILE.
		\newline {\tt erwin.topp@usach.cl}
	}

	\begin{abstract} 
		In this paper we provide a rate of convergence for periodic homogenization of  Hamilton-Jacobi-Bellman equations with nonlocal diffusion. The result is based on the regularity of the associated effective problem, where the convexity plays a crucial role. Such regularity estimates are possible from the available literature once we provide a representation formula for the effective Hamiltonian, a result that has an independent interest.
	\end{abstract}

	\maketitle
	
    
    \noindent\textbf{AMS classification (2020):} 35D40, 35J60, 35R09, 35B27.
	\section{Introduction}
        
        In this paper we are interested in the periodic homogenization of stationary Hamilton-Jacobi-Bellman equations of the form
        	\begin{align}\label{eq}
            	u^\epsilon + H(x, \nicefrac{x}{\epsilon}, Du^\epsilon(x), u^\epsilon) = 0 \quad \mbox{for} \ x \in \R^N,
        	\end{align}
    	where $\epsilon \in (0,1)$. The Hamiltonian $H$ in the equation above has a nonlocal dependence in the last variable, and for this reason we start this note describing its structure.
    	
    	Denote by $\mathbb T^N$ the $N$-dimensional flat torus, and let $\Theta$ be a compact metric space. We consider continuous and bounded functions
        	\begin{equation*}
            	f: \R^N \times \mathbb T^N \times \Theta \to \R^N; \quad l : \R^N \times \mathbb T^N \times \Theta \to \R,
        	\end{equation*}
    	and for $x, p \in \RN, y \in \TN$ and $\varphi \in C^2_b(\R^N)$, we write
    		\begin{equation}\label{H}
            	H(x,y, p, \varphi) = \sup_{\theta} \{ -L_{y}^\theta \varphi(x) -f^{\theta}(x,y)\cdot p - l^{\theta} (x,y) \},
        	\end{equation}
    	where we have adopted the notation $f^\theta(x,y) = f(x,y,\theta)$, and similarly for $l$.
        	
       	We fix $\sigma \in (1,2)$---the ``order" of the operator $L$---and consider a function $A: \mathbb{T}^N \times S^{N-1} \times \Theta \to \mathbb S^N$, where $S^{N - 1}$ is the unit sphere in $\R^N$ and $\mathbb S^N$ is the set of $N \times N$ symmetric matrices. Thus, we have at disposal a family of kernels $\{ K^\theta \}_{\theta \in \Theta}$ defined as
        	\begin{equation}\label{defK}
            	K^\theta(y,z) = |z^t A^\theta(y, \hat{z}) z|^{-(N + \sigma)/2}, \quad y \in \mathbb{T}^N, \ z \in \R^N \setminus \{ 0 \},
        	\end{equation}
        where $\hat{z}= \nicefrac{z}{|z|}$, with $K^\theta(y,z) = K^\theta(y,-z)$, and through them, for each $\theta \in \Theta$, $x, y \in \R^N$ and $\varphi\in C_b^2(\RN)$, we define
        	\begin{equation}\label{nonloc_L}
            	L_{y}^\theta \varphi(x) = \mathrm{P.V.} \int_{\RN} [\varphi(x+z) - \varphi(x)] K^\theta (y,z)  \,dz,
        	\end{equation}
       	where $P.V.$ stands for the Cauchy Principal Value. When $A^\theta(y, \hat z) = c_\sigma I_N$ for all $y\in\mathbb{T}^N$, $I_N$ is the identity matrix, and $c_\sigma > 0$ is an adequate normalizing constant, $L^\theta$ equals $\Delta^{\sigma/2}$, the nowadays well-known fractional Laplacian of order $\sigma$. See~\cite{di2012hitchhikers} for a complete review of this operator. Notice that the kernels $K$ defining $L$ are homogeneous of degree $N + \sigma$, that is $K^\theta(y, \lambda z) = \lambda^{-(N + \sigma)} K^\theta(y, z)$ for all $\theta, y$, $z \neq 0$ and $\lambda > 0$.
       	
       	\medskip
       	
       	        Having described the structure of~\eqref{eq}, we are in position to reveal our main concerns: establishing the convergence of the family of (viscosity) solutions $\{ u^\epsilon \}_\epsilon$ of \eqref{eq} to a function $\bar u \in C(\R^N)$, and providing a rate for the convergence $u^\epsilon\to \bar{u}$; i.e., an estimate of the form $\|u^\epsilon - \bar{u}\|_\infty = \omega(\epsilon)$, where $\omega$ is an explicit modulus of continuity.
       	        In the procedure, we are also interested in the study of the limit function $\bar u$ as the solution of an ``average" fractional PDE, the so-called \textsl{effective problem}.
       	
       	\medskip
       	
       	Fractional PDEs have attracted the attention of the community in the last decade because of its wide number of applications. For instance, problem~\eqref{eq} can be regarded as the dual formulation of an optimal control problem of jump diffusion processes. In the case $A$ is independent of $z$ in~\eqref{defK} and unitary, denoting $\tilde A = \tilde A^\theta(y)$ such that $A = \tilde A \tilde A^t$ (the ``square root" of $A$), we have the nonlocal operator $L$ takes the so-called \textsl{L\'evy-Ito} form
       	$$
       	L^\theta_y u(x) = \mathrm{P.V.} \int_{\R^N} [u(x + \tilde A^\theta(y) z) - u(x)] |z|^{-(N + \sigma)}dz.
       	$$
       	
       	Written in this way, $L$ arises in the Ito formula of a general L\'evy process as the \textsl{jump} component of its generator. Now, we consider the controlled SDE
       	\begin{equation*}
       	dX_t = f(X_t, \epsilon^{-1} X_t, \theta_t) dt + \int_{\R^N} \tilde A(\epsilon^{-1} X_s, \theta_s) z \tilde N(dt, dz),
       	\end{equation*}
       	where for a given filtered probability space, $\tilde N$ is the compensated Poisson random measure of the isotropic $\sigma$-stable L\'evy process (whose infinitesimal generator is the fractional Laplacian of order $\sigma$), and where $\theta_t$ is an adapted process with values in $\Theta$. Then, for each $x \in \R^N$, the value function
       	$$
       	u^\epsilon(x) = \inf_{\theta(\cdot)} \mathbb{E} \int_{0}^{\infty} e^{-s} \ell(X_s, \epsilon^{-1} X_s, \theta_s) ds,
       	$$
       	is a viscosity solution to~\eqref{eq}, see~\cite{OS, Pham}. Thus, the homogenization phenomena observed as $\epsilon \to 0$ in~\eqref{eq} are related to the study of the average behavior of this highly oscillatory stochastic optimal control problem. Other applications related to nonlocal problems can be found in the literature, such as dislocation dynamics \cite{imbert2008homogenization} and front propagation of reaction-diffusion equations \cite{souganidis2019front}.

        \medskip

        Periodic homogenization results for first- and second-order degenerate elliptic Hamilton-Jacobi equations have a long history and have constituted an active research field since the appearance of the pioneering work of Lions, Papanicolaou and Varadhan~\cite{LPV}. A (non-exhaustive) list of notable contributions to the field developed since then includes~\cite{evans1989perturbed, evans1992periodic, Ab01, concordel1996periodic, barles2007some}. We also highlight recent works on stochastic homogenization (see, e.g., \cite{souganidis1999stochastic, lions2010stochastic, caffarelli2010rates, armstrong2012stochastic}). For an introduction to homogenization results in elliptic PDEs with variational structure, where a different set of tools is available, we refer the reader to the books \cite{cioranescu1999introduction, tartar2009general}.

         In the nonlocal setting, periodic homogenization results have been obtained by working roughly within the framework outlined above, see e.g.~\cite{A09, A12, schwab2010periodic, Bct19, ciomaga2020periodic, kassmann2019homogenization, piatnitski2017periodic}. However, to the best of our knowledge, the rate of convergence for the homogenization of nonlocal equations has not yet been addressed. For expository reasons, we proceed to explain the main ideas behind this pursuit, beginning with some well-known facts, and refer the reader to Section~\ref{sec1} for precise assumptions on the data in~\eqref{eq} and the definition of some auxiliary concepts.
        
        \medskip

        Previous results for fractional equations show that the study of the cell problem associated to \eqref{eq} is possible following the scheme of Lions, Papanicolaou and Varadhan. This leads to the existence of an effective Hamiltonian $\bar H$, which in this case can be regarded as a function $\bar H : \RN \times \RN \times C_b^2(\RN) \to \R$. 
        This function is defined through an eigenvalue problem that we describe next: for each $x, p \in \R^N$ and $\phi \in C_b^2(\RN)$, there exists a unique constant $c \in \R$ such that
            \begin{equation}\label{cp}
                \sup_\theta \{ -L^\theta_y w(y) - L^\theta_y \phi(x) - f^\theta(x,y)\cdot p - l^{\theta}(x,y) \} = c\quad\text{for} \  y \in \mathbb{T}^N,
            \end{equation}
        has a viscosity solution $w \in C(\TN)$. This equation is referred as the \textsl{cell problem} associated to $H$, and $c$ is the \textsl{ergodic constant}, see for instance~\cite{BChCI} for the solvability of such a problem. This allows us to define $\bar H(x,p,\phi) := c$. 
        
        In spite of not having closed formulae for $\bar H$ in most cases, it is possible to prove that this function \emph{inherits} various properties from $H$, such as continuity and/or convexity. This is obtained following ideas already in \cite{evans1992periodic}, and in the nonlocal setting it is done in e.g.~\cite{ciomaga2020periodic}. Another property satisfied by $\bar H$ in the present nonlocal context is the \emph{degenerate ellipticity}, intended in the following sense: given $x, p \in \RN$ and $\varphi_1, \varphi_2 \in C^2_b(\RN)$ such that $\varphi_1(x) = \varphi_2(x)$ and $\varphi_1 \geq \varphi_2$ in $\RN$, then
            \begin{equation}\label{barHmon}
                \bar H(x,p,\varphi_1) \leq \bar H(x,p,\varphi_2).
            \end{equation}

        Following the classical \textsl{perturbed test function method} of Evans~\cite{evans1989perturbed, evans1992periodic}, it is possible to prove that the sequence of solutions $u^\epsilon$ to~\eqref{eq} converges as $\epsilon \to 0$ (up to subsequence) to the solution of the \textsl{effective} problem
            \begin{equation}\label{eqnonloceff}
                \bar u + \bar H(x, D\bar u(x), \bar u) = 0 \quad \mbox{for} \ x \in \RN.
            \end{equation}
  	  
  	  Thus, it is a well-known fact that uniqueness for the limit problem~\eqref{eqnonloceff} implies the full convergence of the sequence $(u^\epsilon)$, leading to homogenization.

        \medskip
        
        In general---for instance, when the Hamiltonian is not translation-invariant---the monotonicity condition~\eqref{barHmon} is not sufficient to compare discontinuous viscosity sub- and supersolutions of~\eqref{eqnonloceff}. This impedes the use of the half-relaxed limits method of Barles and Perthame~\cite{Bp87}, an efficient tool used in the first and second-order setting to obtain homogenization. Though partial results on the comparison principle for non-translation-invariant fractional equations exist, they are either related to L\'evy-Ito type operators, or require some regularity of the sub- or supersolution being compared; see e.g.~\cite{barles2008second, GMS19, Bct19, mou2015uniqueness}. Nevertheless, condition~\eqref{barHmon} by itself is sufficient to compare classical solutions of~\eqref{eqnonloceff}. This motivates the study of $C^{\sigma+\alpha}$-estimates for the effective equation.
        
       	Once $C^{\sigma+\alpha}$-estimates for~\eqref{eqnonloceff} are at hand, they can be employed to obtain the result on rate of convergence, following the approach presented by Camilli and Marchi in~\cite{camilli2009rates} for second-order convex, fully nonlinear equations. The argument relies on the study of the difference $u^\epsilon - \bar u$ using the classical doubling variables method in the viscosity solution's theory. Since $u^\epsilon$ and $\bar u$ solve different problems, the usual penalization procedure requires the consideration of the \textit{corrector function} $w$ solving~\eqref{cp} to modulate the difference among the original problem and the effective one. Actually, we need to consider the solution $w_\lambda$ of the approximating problem~\eqref{dp} below (the so-called \textsl{vanishing discount problem}), which satisfies $w_\lambda \to w$ as $\lambda \to 0$, using the last convergence in an adequate regime depending on $\epsilon$. This is the content of our main Theorem~\ref{teoconv} below, from which homogenization is a byproduct of such a  rate of convergence.

 \medskip
        
        The mentioned regularity estimates are made possible in our setting by representation formulas for the effective Hamiltonian $\bar H$ that have an interest in their own right. These formulas are contained in Theorem \ref{corbarH}, which in turn is a direct consequence of Proposition \ref{propc}. The proof of this last proposition follows the lines of Ishii, Mitake and Tran \cite{IMT}, which characterizes the ergodic constant associated to second-order Hamilton-Jacobi-Bellman problems.
        This allows us to conclude an ``average" formula for $\bar H$ in analogy to the characterization of the (multiplicative) principal eigenvalue in the celebrated paper of Donsker and Varadhan~\cite{DV76} for linear second-order equations; see also Armstrong~\cite{A09} for the fully nonlinear second-order case. 
        
                The problem of obtaining such a representation formula for an operator like $\bar{H}$ dates back to the work of Courrège \cite{courrege1965forme}, in which linear operators satisfying a \emph{global} comparison principle are characterized of being of \textit{Lèvy type}. More recently, this study has been extended to the nonlinear setting in \cite{guillen2019min,guillen2020min}, in even greater generality than is necessary for the problem at hand (see also \cite{schwab2010periodic, guillen36neumann}). Nevertheless, these formulas are not sufficient to conclude comparison principles and/or regularity estimates for equation~\eqref{eqnonloceff}, at least in the most general setting presented in this paper. 
                
                We finish mentioning that our representation formula is valid in greater generality, for instance when the nonlocal operator also depends on the slow variable $x$. However, the conclusion regarding regularity for the effective problem and its subsequent application on the rate of convergence in the homogenization procedure requires further analysis in this case, and we do not pursue in this direction here. See the discussion about this issue in Subsection~\ref{rmk_xDependence}.

\medskip

The paper is organized as follows. In Section~\ref{sec1} we provide the assumptions and present the main results of the paper. In Section~\ref{sec_eff} we prove the representation formula for the effective Hamiltonian. In Section~\ref{secdp} we provide relevant estimates of the discount problem, the approximating equation that defines the effective Hamiltonian. Finally, in Section~\ref{secrate} we provide the rate of convergence for our homogenization problem.
        
    \section{Assumptions and Main results}
    \label{sec1}
    
          Throughout this note, $H$ will have the form~\eqref{H}. We assume $f \in C(\RN\times\TN\times\Theta; \RN)$ and $l \in C(\RN\times\TN\times\Theta; \R)$ and there exists a constant $C$ such that 
    \begin{equation}\label{lip_data}
    |Df^\theta|, |Dl^{\theta}|\ \leq C, \quad\textrm{for all }\theta\in \Theta,
    \end{equation}
    for some $C>0$, in the viscosity sense; in other words, both functions are Lipschitz continuous with respect to both variables, uniformly in $\theta$.
    
    The nonlocal operator $L$ has the form~\eqref{nonloc_L}, with $A \in C(\TN \times S^{N - 1} \times \Theta; \mathbb{S}^N)$ satisfying
    \begin{equation}\label{holder_k}
    \|A^\theta\|_{C^\alpha(\mathbb{T}^N \times S^{N - 1})} \leq C_0
    \end{equation}
    for some $C_0>0$, uniformly with respect to $\theta$, and is uniformly elliptic in the sense that there exists constants $0 < \gamma \leq \Gamma$ such that
    \begin{equation}\label{A_elliptic_order}
    \Gamma^{-2/(N + \sigma)} \leq A^\theta \leq \gamma^{-2/(N + \sigma)}
    \end{equation}
    in the sense of matrices. 
    
    In particular, the family of kernels $\{ K^\theta \}$ satisfies 
    \begin{equation}\label{elliptic_order}
    \frac{\gamma}{|z|^{N+\sigma}} \leq K^\theta(y,z)\leq \frac{\Gamma}{|z|^{N+\sigma}} \quad\textrm{for all } y\in \mathbb{T}^N, z\in \RN\backslash\{0\},
    \end{equation}
    in addition to being $\mathbb{T}^N$-periodic in $y$.

\bigskip
    
        Now we present the first main result of this paper, which is a representation formula for $\bar H$. Given $L$ as above, define the set
            \begin{equation}\label{G0}
                \mathcal G_0 = \{ \phi \in C(\mathbb{T}^N \times \Theta) \ : \ \exists \ u \ \mbox{s.t.} \ \sup_{\theta \in \Theta} \{ -L^\theta_y u(y) - \phi_\theta(y) \} \leq 0 \ \mbox{for} \ y \in \mathbb{T}^N \},
            \end{equation}
        where the equation satisfied by $u$ is understood in the viscosity sense. As it can be seen in Lemma 2.8 in~\cite{IMT}, $\mathcal G_0$ is a nonempty convex cone with vertex at the origin. We also
        denote $\mathcal P$ the space of probability measures on $\mathbb{T}^N \times \Theta$, and consider the polar cone $\mathcal G_0'$ given by
            \begin{equation}\label{G0'}
                \mathcal G_0' := \left\{ \mu \in \mathcal P \ : \ \int \phi d\mu \geq 0 \ \mbox{for all} \ \phi \in \mathcal G_0 \right\}.
            \end{equation}

            \begin{theorem}\label{corbarH}
                Assume $\{ K \}$ has the form~\eqref{defK}, and $f,l$ are continuous with respect to all its variables, and that~\eqref{lip_data}, \eqref{holder_k} and~\eqref{A_elliptic_order} hold. Let $\mathcal G_0'$ be defined in~\eqref{G0'}. Then, for all $x, p \in \R^N$ and $\phi \in C^{\sigma + \iota}$, the function $\bar H$ defined in~\eqref{cp} has the form
                   	\begin{equation*}
                       	\overline{H}(x, p, \phi) = \sup_{\mu \in \mathcal{P} \cap \mathcal G_0'} \left\{ -\overline{L}^\mu \phi(x) - \bar{f}^\mu(x)\cdot p - \bar{l}^\mu(x)\right\},
                   	\end{equation*}
                where, for each $\mu \in \mathcal P$ we have denoted
                   	\begin{align*}
                       	\overline{L}^\mu \phi(x) & := \int_{\RN} [\phi(x+z) - \phi(x)] \overline{K}^\mu(z) \,dz; \quad \bar K^\mu(z) = \int_{\mathbb{T}^N \times \Theta} K^\theta(y, z) d\mu, \\
                       	\bar f^\mu (x) & := \int_{\mathbb{T}^N \times \Theta} f(x,y,\theta) d\mu, \\
                       	\bar l^\mu (x) & := \int_{\mathbb{T}^N \times \Theta} l(x,y,\theta) d\mu.
                   	\end{align*}
            \end{theorem}

        \medskip
  
        In spite of $\bar H$ not being translation-invariant, the representation formula obtained in Theorem \ref{corbarH} allows us to establish comparison --and thus, homogenization-- through a superposition of available regularity results which allows us to ensure $C^{\sigma + \alpha}$ estimates for the effective problem, see Corollary \ref{eff_reg}. 
  
        \medskip
     	
        As mentioned earlier, the regularity estimates for the effective problem which stem from Theorem \ref{corbarH} lead to the main result of this paper.
              \begin{theorem}\label{teoconv}
                             	Under the assumptions of Theorem~\ref{corbarH}, there exist constants $C>0$ and $\bar{\alpha}\in (0,1)$, depending on the data and parameters of equation \eqref{eq}, but not on $\epsilon$, such that
                                 	\begin{equation*}
                                     	|u^\epsilon(x) - u(x)| \leq C \epsilon^{\bar{\alpha}} \quad \text{for all }x\in \RN, \epsilon\in (0,1).
                                 	\end{equation*}
                             	where $u^\epsilon$ and $u$ are the solutions of \eqref{eq} and \eqref{eqnonloceff}, respectively.
              \end{theorem}

        For completeness, we include also a (much easier) result on the rates of convergence for the homogenization of \eqref{eq} when dependence on the slow variable---and possibly the gradient of the solution---is dropped. The corresponding results in the first- and second-order setting are contained in \cite{capuzzo2001rate} and \cite{camilli2009rates}, respectively.
            \begin{theorem}\label{teo_simple_rate}
                  Under the assumptions of Theorem \ref{teoconv},
                      \begin{enumerate}[(i)]
                          \item if $f^\theta\equiv 0$ and $l^{\theta}$ depends only on $y\in \mathbb{T}^N$ (i.e., $l^{\theta}(x,y)=l^{\theta}(x))$ for all $\theta\in \Theta$, then there exists a constant $C>0$ such that
                              \begin{equation*}
                                  |u(x) -u^\epsilon(x)| \leq C \epsilon^\sigma \quad \textrm{for all }x\in \RN,\ \epsilon\in (0,1).
                              \end{equation*}
                          \item if $f^\theta$ and $l^{\theta}$ depend only on $y\in \mathbb{T}^N$ (i.e., $f^\theta(x,y)= f^\theta(x)$, $l^{\theta}(x,y)=l^{\theta}(x))$ for all $\theta\in \Theta$, then there exists a constant $C>0$ such that
                              \begin{equation*}
                                  |u(x) -u^\epsilon(x)| \leq C \epsilon^{\sigma-1} \quad \textrm{for all }x\in \RN,\ \epsilon\in (0,1).
                              \end{equation*}
                      \end{enumerate}
            \end{theorem}
                 
    \section{Representation formula: Proof of Theorem~\ref{corbarH}}\label{sec_eff}
        Theorem \ref{corbarH} will be obtained as a consequence of Proposition \ref{propc} below, which provides two characterizations of the ergodic constant $c$ in~\eqref{ergodic}. We present the result in a slightly more general setting: we consider a measurable function $K: \TN \times \R^N \times \Theta \to \R$ such that $K^\theta(y,z) = K^\theta(y,-z)$ for all $y \in \TN, z \in \R^N, \theta \in \Theta$, and satisfying the ellipticity condition~\eqref{elliptic_order}. We consider 
        \begin{equation}\label{opcp}
        L_y^\theta u(y) = \mathrm{P.V.} \int_{\R^N} [u(y + z) - u(y)] K^\theta(y,z)dz.
        \end{equation}
        
        Let $\ell \in C(\mathbb{T}^N \times \Theta)$ and consider the following ergodic problem: find a pair $(\psi, c) \in C(\TN) \times \R$ solving the equation
                \begin{equation}\label{ergodic}
                    F(\psi, y) := \sup_{\theta} \{ - L^\theta_{y} \psi(y) - \ell^{\theta}(y) \} = c \quad \mbox{in} \ \mathbb{T}^N,
                \end{equation}
                in the viscosity sense. Under mild assumptions on $K$, it is possible to prove the existence of a unique constant $c \in \R$ for which problem~\eqref{ergodic} has a viscosity solution, and this solution is in $C^{\sigma + \alpha}$ for some $\alpha \in (0,1)$, see~\cite{CSAnnals, S}. These assumptions are related with the continuity on the data and ellipticity conditions that ensure comparison principles and regularity, see for instance~\cite{BChCI, BLT}.
        
        \medskip
        
              \begin{proposition}\label{propc}
                  Assume $\ell \in C(\mathbb{T}^N \times \Theta)$, $L$ has the form~\eqref{opcp} with $K$ symmetric in $z$, satisfying~\eqref{elliptic_order} and such that the ergodic problem~\eqref{ergodic} has a solution $(\psi, c)$. Let $\mathcal G_0'$ defined as in~\eqref{G0} with these operators $L$. Then, $c$ can be characterized as follows
                      \begin{equation}\label{twocharact}
                          \begin{split}
                              c & = -\inf_{\mu \in \mathcal P} \sup_{\psi \in C^{\sigma + \iota}} \int_{\mathbb{T}^N \times \Theta} \{L^\theta_y \psi (y) + \ell^{\theta}(y) \} d\mu \\
                              & = -\inf_{\mu \in \mathcal G_0'} \int_{\mathbb{T}^N \times \Theta} \ell^{\theta}(y) d\mu.
                          \end{split}
                      \end{equation}
              \end{proposition}
     	
     	\begin{proof}
     	Using the comparison principle for the associated parabolic problem (see Proposition 3.2 in~\cite{tabet2010large}), which is possible in our case due to the smoothness of the solution to~\eqref{ergodic}, it is possible to prove that $c$ can be characterized as
       	    \begin{equation}\label{caract1}
             	c = \inf \{ \tilde c \in \R : \exists \ \psi \ \mbox{s.t.} \ F(\psi, y) \leq \tilde c \ \mbox{in} \ \mathbb{T}^N \},
         	\end{equation}    
        where the inequality inside the $\inf$ is understood in the viscosity sense. From here, we claim that $c$ meets the value
            \begin{equation*}
                \inf_{\psi \in C^{\sigma + \iota}} \sup_{y \in \mathbb{T}^N} F(\psi, y).
            \end{equation*}
      
        Indeed, let us denote the latter value by $\bar c$. Since there exists a smooth solution to~\eqref{ergodic}, we clearly have $\bar c \leq c$. On the other hand, if $\bar c < c$, there exists $\epsilon > 0$ small and $\psi$ smooth such that
            \begin{equation*}
                F(\psi, y) \leq \bar c + \epsilon < c \quad \mbox{for all} \ y \in \mathbb{T}^N,
            \end{equation*}
        which contradicts the characterization~\eqref{caract1}.
      
     	Using this last formula for $c$ and the continuity of the map 
         	$$
             	(y, \theta) \mapsto -L^\theta_{y} \psi(y) - \ell^{\theta}(y)
         	$$ 
     	when $\psi$ is smooth, we conclude that
         	\begin{equation*}
             	c = -\sup_{\psi \in C^{\sigma + \iota}} \inf_{\mu \in \mathcal P} \int_{\mathbb{T}^N \times \Theta} \{L^\theta_{y} \psi (y) + \ell^{\theta}(y) \} d\mu.
         	\end{equation*}
     	Note that for each $\psi$ smooth, the map 
       	    $$
             	\mu \mapsto \int_{\mathbb{T}^N \times \Theta} \{L^\theta_{y} \psi(y) + \ell^{\theta}(y) \} d\mu
         	$$ 
        is a bounded linear map when the $*$-weak topology is considered on $\mathcal P$. On the other hand, given $\mu \in \mathcal P$, the map
            $$
                \psi \mapsto \int_{\mathbb{T}^N \times \Theta} \{L^\theta_{y} \psi(y) + \ell^{\theta}(y) \} d\mu
            $$
        is affine (hence, concave). Then, using Sion's Theorem~\cite{Sion}, we get that $c$ can be written as
            \begin{equation}\label{caract2}
                -c =  \inf_{\mu \in \mathcal P} \sup_{\psi \in C^{\sigma + \iota}} \int_{\mathbb{T}^N \times \Theta} \{L^\theta_{y} \psi(y) + \ell^{\theta}(y) \} d\mu,
            \end{equation} 
        which is the first characterization in~\eqref{twocharact}.
      
        For the second, we denote
            \begin{equation*}
                \bar c = \inf_{\mu \in \mathcal G_0'} \int_{\mathbb{T}^N \times \Theta} \ell^{\theta}(y) d\mu.
            \end{equation*}
        Note that if $\psi \in C^{\sigma + \iota}(\mathbb{T}^N)$, the function $(y, \theta) \mapsto -L^\theta_{y} \psi(y)$ is in $\mathcal G_0(x)$. Then, by~\eqref{caract2} we immediately see that $-c \leq \bar c$.
      
        On the other hand, using that $\mathcal G_0$ is a cone with vertex at the origin, for each $\mu \notin \mathcal G_0'$ we have $\inf_{\phi \in \mathcal G_0} \int \phi d\mu = -\infty$; whereas, if $\mu \in \mathcal G_0'$ we have $\inf_{\phi \in \mathcal G_0} \int \phi d\mu = 0$. Thus, we can write
            \begin{align*}
                \bar c ={}& \inf_{\mu \in \mathcal P} \left\{ \int_{\mathbb{T}^N \times \Theta} \ell^\theta(y) d\mu - \inf_{\phi \in \mathcal G_0} \int_{\mathbb{T}^N \times \Theta} \phi^\theta(y) d\mu \right\}\\
                 ={}& \inf_{\mu \in \mathcal P} \sup_{\phi \in \mathcal G_0} \left\{ \int_{\mathbb{T}^N \times \Theta} \ell^\theta - \phi^\theta d\mu \right\}.
            \end{align*}
      
        Invoking again Sion's Theorem, we conclude that
            \begin{equation*}
                \bar c =  \sup_{\phi \in \mathcal G_0} \inf_{\mu \in \mathcal P} \left\{ \int_{\mathbb{T}^N \times \Theta} \ell^\theta - \phi^\theta d\mu \right\}.
            \end{equation*}
     	Then, if by contradiction we assume that $-c < \bar c$, there exists $\epsilon > 0$ small and $\phi \in \mathcal G_0$ such that
         	\begin{equation*}
             	-c + \epsilon < \int_{\mathbb{T}^N \times \Theta} \ell^\theta - \phi^\theta d\mu \quad \mbox{for all} \ \mu \in \mathcal P.
         	\end{equation*}
     	
         Since the \emph{Dirac deltas} are in $\mathcal P$, in particular we have
         	\begin{equation*}
             	-c + \epsilon < \ell^{\theta}(y) - \phi^\theta(y) \quad \mbox{for all} \ (y, \theta) \in \mathbb{T}^N \times \Theta.
         	\end{equation*}    
        Thus, by definition of $\mathcal G_0$, there exists $u$ such that
            \begin{equation*}
                - L^\theta_{y} u(y) -\ell^{\theta}(y) < c - \epsilon \quad \mbox{for all} \ (y, \theta) \in \mathbb{T}^N \times \Theta,
            \end{equation*}
        but this contradicts the characterization of $c$ in~\eqref{caract1}. This concludes the second characterization in~\eqref{twocharact}.
   	\end{proof}

    \begin{proof}[Proof of Theorem \ref{corbarH}]
        We again assume that $K^\theta$ has the form \eqref{defK}---i.e., neither $K^\theta$ nor the sets $\mathcal{G}_0, \mathcal{G}_0'$ depend on the the slow variable $x$.
        
        As described in the introduction, $\bar{H}$ is defined as the ergodic constant in \eqref{cp}. Therefore, for given $x, p \in \R^N$ and $\phi \in C_b^2(\RN)$, we apply the second characterization in \eqref{twocharact} with $\ell^\theta(y) = L^\theta_y \phi(x) + f^\theta(x,y)\cdot p + l^{\theta}(x,y)$, and conclude by integrating over $\mathbb{T}^N \times \Theta$.
    \end{proof}
    
    \medskip
    
    As we mentioned in the Introduction, we have the following regularity result for the solution of the effective problem~\eqref{eqnonloceff}.
        \begin{corollary}\label{eff_reg}
          	Assume hypotheses of Corollary~\eqref{corbarH} holds, and assume further that $f, l$ are H\"older continuous in the slow variable, uniformly with respect to the rest of the others, Then, every continuous solution $u$ of \eqref{eqnonloceff} is in $C^{\sigma+\alpha}$, and there exist $C>0$ and $\alpha>0$, depending only on universal constants, such that $\|u\|_{\sigma+\alpha} \leq C$. Moreover, this solution is unique in the class of smooth functions.
        \end{corollary}

    The proof follows by available regularity results. First, $C^\alpha$ estimates for~\eqref{eq} which are uniform in $\epsilon$ (see Chang-Lara and D\'avila~\cite{ChLD16}), and the perturbed test function method, allows us to prove that each accumulation point of the sequence is a continuous viscosity solution to~\eqref{eqnonloceff}. Such a solution becomes Lipschitz continuous using the results of Barles, Chasseigne, Ciomaga and Imbert~\cite{BChCILip}, and from here we get it is $C^{1,\alpha}$ by Caffarelli and Silvestre~\cite{caffarelli2009regularity}. Finally, the $C^{\sigma+\alpha}$ estimates in~\cite{CSAnnals} or Serra~\cite{S} imply the desired higher-order regularity. This procedure is the content of the regularity estimates presented in Corollary~\ref{eff_reg}. We remark that such a procedure cannot be directly applied to~\eqref{eq} to obtain $C^{\sigma+\alpha}$ estimates independent of $\epsilon$, due to the high oscillatory phenomena as $\epsilon \to 0$ which deteriorates the continuity of the data in the equation.
      
        \subsection{A discussion on the dependency of $\bar H$ on the slow variable.}\label{rmk_xDependence}
    We notice here that the arguments of Proposition \ref{propc} and Theorem \ref{corbarH} apply equally to the more complicated case of $x$-dependent kernels, for instance, for kernels with the form
                    \begin{equation}\label{KxDep}
                            K^\theta(x,y,z) = |z^t A(x,y, \hat z) z|^{-(N + \sigma)/2},
                        \end{equation}
            with certain continuity assumptions on the slow variable $x$.
             In that case we still produce a representation formula, but in that case the new index set depends upon $x$, namely
                \begin{equation}\label{barHx}
                    \overline{H}(x, p, \phi) = \sup_{\mu \in \mathcal G_0'(x)} \left\{ -\overline{L}^\mu \phi(x) - \bar{f}^\mu(x)\cdot p - \bar{l}^\mu(x)\right\}.
                \end{equation}
                
                      	However, we notice that the presence of $\mathcal G_0'(x)$ in~\eqref{barHx} does not allow us to use the available regularity results to conclude an analogue of Corollary \ref{eff_reg}. 

            As the method to obtain rates of convergence in the following section relies on such estimates, we are unable to reach a conclusion in the setting imposed by~\eqref{KxDep}.
            
            Nevertheless, we highlight that the structure of~\eqref{barHx} is in accordance with phenomena arising, for instance, in first-order equations. For Hamiltonians of the form
               	\begin{equation*}
                   	H(x, y , p) = \sup_{\theta \in \Theta} \{ - f^\theta(x, y) \cdot p - l^{\theta}(x,y) \},
               	\end{equation*}
           	under certain controlability assumptions on the family of fluxes $f$, the associated effective Hamiltonian takes the form
               	\begin{equation*}
                   	\bar H(x, p) = \sup_{\mu \in Z(x)} \{ - f^\mu(x) \cdot p - l^\mu(x) \},
               	\end{equation*}
          	where for each $x \in \RN$, $Z(x)$ is a subset of the space of probability measures on $\TN \times \Theta$---the set of ``occupational measures", as described in~\cite{Terrone} and references therein. In this setting, some continuity on this set of indices $Z(x)$ (in the Hausdorff distance) is required in the standard doubling-variables method to get comparison principles for viscosity solutions. This makes possible to identify the effective problem with a optimal control problem with average trajectories and costs, see Bardi and Terrone in~\cite{BT15}.
          	
          	\medskip

             The case of nonlocal terms of order $\sigma\in (0,1]$ contains similar difficulties related to those of the previous point. First, we observe that the structure of the cell problem changes depending on the value of $\sigma$, as it is shown in \cite{Bct19}. For the critical value $\sigma=1$---the case which is studied in detail in \cite{ciomaga2020periodic}---the scaling of the problem makes the cell problem \eqref{cp} depends also on the gradient of the corrector $w$. Consequently, the inequality defining $\mathcal{G}_0$ in \eqref{G0} is now
                \begin{equation*}
                    \sup_{\theta \in \Theta} \{ -L^\theta_y v(y) - f^\theta(x,y)\cdot Dv(y) - \phi_\theta(y) \} \leq 0 \ \mbox{in} \ \mathbb{T}^N.
                \end{equation*}      
                
            This implies that the sets $\mathcal{G}_0$ and $\mathcal{G}_0'$ both depend on the slow variable $x$, even if $K^\theta$ has the simpler form \eqref{defK}, and we are again in the situation described in the previous remark.
     
  \section{Estimates for the Discount Problem}\label{secdp}
  
  In the study of the cell problem \eqref{cp} it is convenient to consider the approximation problem 
  \begin{equation}\label{dp}
  \lambda w_\lambda + \sup_{\theta} \{ -L^\theta_y w_\lambda(y) - L^\theta_y \phi(x) - f^\theta(x,y)\cdot p - l^{\theta}(x,y) \} = 0\quad\text{in }\mathbb{T}^N,
  \end{equation}
  for $\lambda>0$, commonly referred to as the \textit{vanishing discount problem:} it is well-known fact that under the assumptions considered here, the solvability of the eigenvalue problem~\eqref{cp} is obtained  in the passage to the limit as $\lambda \to 0^+$ in~\eqref{dp}. This problem plays a key role in our main result Theorem~\ref{teoconv}.
  
  To stress the dependence of $w_\lambda$ on $x, p$ and $\phi$, we write $w_\lambda = w_\lambda(y; x, p, \phi)$ for the solution of \eqref{dp} and, similarly, $w=w(y; x,p,\phi)$ for the solution of \eqref{cp}.
     	
     	\begin{lemma}\label{lemma_dp_bounds}
     		There exists a constant $C_1$ such that the solution of \eqref{dp} satisfies the following: for all $x, p \in \RN$ and $\phi\in C^{\sigma + \iota}(\RN)$
     		
     		\medskip
     		
     			\begin{enumerate}[(a)]
     				\item\label{dp_Linfty} $\|w_\lambda(\cdot; x, p, \phi)\|_\infty  \leq \lambda^{-1}C_1\left(1 + |p| + \|\phi\|_{\sigma+\iota}\right);$
     				
     				\medskip
     				
     				\item\label{dp_sig+al} for some $\alpha\in (0,1)$,
                      \begin{equation*}
                          \|w_\lambda(\cdot; x, p, \phi) - w_\lambda(0; x, p, \phi)\|_{C^{\sigma+\alpha}(\RN)} \leq C_1 \left(1 + |p| + \|\phi\|_{\sigma+\iota}\right)
                      \end{equation*}
                      
                      \medskip
                      
     				\item\label{dp_diff} $|D_p w_\lambda| \leq \lambda^{-1} C_1$, $|D_x w_\lambda| \leq \lambda^{-1} C_1(1 + |p| + \| \phi\|_{\sigma + \iota})$ (in the viscosity sense); and if $\phi_i \in C^{\sigma+\iota}(\RN)$ for $i=1, 2$, then 
                      \begin{equation*}
                          \|w_\lambda(\cdot; x, p, \phi_1) - w_\lambda(\cdot; x, p, \phi_2)\|_\infty \leq \lambda^{-1} C_1 \|\phi_1 - \phi_2\|_{\sigma+\iota}
                      \end{equation*}
                      
                      \medskip

     				\item\label{dp_van_disc_rate} for all $y\in \mathbb{T}^N, \quad |w_\lambda(y; x, p, \phi) + \overline{H}(x,p, \phi)| \leq \lambda C_1\left(1 + |p| + \|\phi\|_{\sigma+\iota}\right)$.
     			\end{enumerate}
     	\end{lemma}
      
      \begin{proof}

          \textit{(\ref{dp_Linfty})} From the structure of $H$, we infer that $C = \pm \lambda^{-1}C_1\left(1 + |p| + \|\phi\|_{C^{\sigma + \iota}(\RN)}\right)$ are respectively super- and subsolutions of \eqref{dp}. Indeed, it is immediate that
              \begin{equation*}
                  L^\theta_y C \equiv 0,\quad |f^\theta(x,y)\cdot p|\leq  C_1|p|, \quad |l^{\theta}(x,y)|\leq C_1,
              \end{equation*}
         for sufficiently large $C_1$. For the principal part, we have
              \begin{equation*}
                  |L^\theta_y \phi(x)| \leq   |L^\theta_y[B]\phi(x)| + |L^\theta_y[\RN\backslash B]\phi(x)| =: I_1 + I_2.
              \end{equation*} 
          Using the symmetry of $K^\theta$, we have
              \begin{align*}
                  I_1 ={}& \left|\int_{B} \left(\phi(x+ z) - \phi(x) - D\phi(x)\cdot z\right) K^\theta(y,z) \,dz \right| \\
                  ={}& \left|\int_{B} \left(\int_{0}^{1} D\phi(x + tz)\cdot z \,dt\right) - D\phi(x)\cdot z\ K^\theta(y,z) \,dz \right|\\
                  \leq{}& \int_{B} \left(\int_{0}^{1} |D\phi(x+tz) - D\phi(x)| |z| \,dt \right)\ K^\theta(y,z) \,dz\\
                  \leq{}& \int_{0}^{1} [\phi]_{1,\beta} |z|^{1+\beta} \ K^\theta(y,z) \,dz \leq C\Gamma [\phi]_{1,\beta}.
              \end{align*}
          On the other hand, it is easy to see that $I_2 \leq C\Gamma\|\phi\|_\infty$, and by combining both estimates we conclude.
          
          \textit{(\ref{dp_sig+al})} We must first establish that for all $\lambda, x, p$ and $\phi$ as above,
               \begin{equation}\label{dp_Linfty2}
                   \|w_\lambda(\cdot; x, p, \phi) - w_\lambda(0; x, p, \phi)\|_\infty \leq C_1 \left(1 + |p| + \|\phi\|_{\sigma+\iota}\right).
               \end{equation}
           Assume, on the contrary, that there exist sequences $(\lambda_k)_k$ and $\big((x_k, p_k, \phi_k)\big)_k$ such that $\lambda_k\to 0$  and $w_k = w_{\lambda_k}(\cdot; x_k, p_k, \phi_k)$ satisfies
               \begin{equation*}
                   \|w_k - w_k(0)\|_\infty \geq k \left(1 + |p_k|+ \|\phi_k\|_{\sigma+\iota}\right).
               \end{equation*}
           
           For $\eta_k =  \|w_k - w_k(0)\|_\infty^{-1}$,  we define $\tilde{w}_k = \eta_k(w_k - w_k(0))$ and note that $\tilde{w}_k$ satisfies $\tilde{w}_k(0)=0$, $\|\tilde{w}_k\|_\infty = 1$ and
               \begin{equation}\label{eq_k}
                   \lambda_k \tilde{w}_k + \lambda_k \eta_k \tilde{w}_k(0) + \sup_\theta\{-L^\theta_y \tilde{w}_k(y) - \tilde{l}^\theta_k(y)\} = 0,
               \end{equation}
           where $\tilde{l}^\theta_k(y)= L^\theta_y \phi_k(x_k) + f^\theta(x_k, y) + l^{\theta}(x_k, y)$. 
           
           By part \textit{(\ref{dp_Linfty})}---again using \eqref{lip_data}---we have that
               \begin{equation*}
                   \eta_k\lambda_k|\tilde{w}_k(0)| + \eta_k |\tilde{l}^\theta_k|_\infty \leq \frac{C}{k}.
               \end{equation*}
          for some $C>0$. We may thus apply the regularity results of \cite{caffarelli2009regularity} to \eqref{eq_k} to find that $(\tilde{w}_k)$ is bounded in $C^{\bar{\iota}}$ for some $\bar{\iota}>0$, uniformly with respect to $k$, and thus converges up to a subsequence to some $\tilde{w}\in C^{\bar{\iota}}(\mathbb{T}^N)$. Passing to the limit in \eqref{eq_k} in the viscosity sense, we find that $\tilde{w}$ is a solution of
               \begin{equation*}
                   \sup_{\theta} \left\{ -L^\theta_y \tilde{w}(y) \right\} = 0 \quad\textrm{in }\mathbb{T}^N.
               \end{equation*}
           Since $\tilde{w}$ is periodic, it achieves its maximum at some point; hence, it is constant by the strong maximum principle. This contradicts the fact that, as $\tilde{w}$ is a limit of $\tilde{w}_k$, we have $\tilde{w}(0) = 0$ and $\|\tilde{w}\|_\infty = 1$, thus establishing \eqref{dp_Linfty2}.
                    
           We now note that $v_\lambda = w_\lambda - w_\lambda(0)$ satisfies
              \begin{equation}\label{v_lambda_eq}
                  \lambda v_\lambda + \sup_\theta \{ -L^\theta_y v_\lambda(y) - \tilde{l}^\theta(y)\} = 0
               \end{equation}
           where $\tilde{l}^\theta(y)= L^\theta_y \phi(x) + f^\theta(x, y) + l^{\theta}(x, y)$, and using \eqref{dp_Linfty2} and the results of \cite{caffarelli2009regularity} we conclude that $\|v_\lambda\|_{C^\alpha}(\mathbb{T}^N) \leq C_1(1 + |p| + \|\phi\|_{\sigma+\iota})$ for some $\alpha>0$, taking a larger $C_1$ if necessary.
           
           We wish to conclude by applying the results of \cite{S} to \eqref{v_lambda_eq}. To this end, we now show that $\tilde{l}^\theta$ as defined above is uniformly bounded in $C^\alpha$.
           
           Let $y_i\in \mathbb{T}^N$, $i=1,2$, $z\in \RN\setminus\{0\}$. Given the assumptions \eqref{lip_data}, it is immediate that 
              \begin{equation*}
                  \|f^\theta(x, \cdot)\cdot p + l^{\theta}(x, \cdot)\|_{C^\alpha(\mathbb{T}^N)} \leq C_1(1 + |p|),
              \end{equation*}
           in fact for any $\alpha\in (0,1]$. To bound the nonlocal term appearing in $\tilde{l}^\theta$, we compute
              \begin{equation*}
                  |K^\theta(y_1,z) - K^\theta(y_2,z)| \leq \frac{\Gamma^2}{|z|^{2(N+\sigma)}} \left| \left(z^TA^\theta(y_1, \hat z)z\right)^{\frac{N+\sigma}{2}} - \left(z^TA^\theta(y_2, \hat z)z\right)^{\frac{N+\sigma}{2}} \right|.
              \end{equation*}
          Write $h(s):=s^{\frac{N+\sigma}{2}}$, $s_i:= z^TA^\theta(y_i, \hat z)z$, $i=1, 2$. Since $h$ is smooth and convex, we have
              \begin{equation*}
                  |h(s_1) - h(s_2)| \leq |s_1-s_2| \max\{ |h'(s_1)|, |h'(s_2)| \}.
              \end{equation*}
          Therefore,
              \begin{align*}
                  &\left| \left(z^TA^\theta(y_1, \hat z)z\right)^{\frac{N+\sigma}{2}} - \left(z^TA^\theta(y_2, \hat z)z\right)^{\frac{N+\sigma}{2}} \right| \\
                  &\quad \leq\left| z^T[A^\theta(y_1, \hat z)-A^\theta(y_2, \hat z)z]\right|\max\left\{ \left(z^TA^\theta(y_1, \hat z)z\right)^{\frac{N+\sigma-2}{2}}, \left(z^TA^\theta(y_2, \hat z)z\right)^{\frac{N+\sigma-2}{2}} \right\}\\
                  &\quad \leq C|y_1-y_2|^\alpha|z|^{N+\sigma},
              \end{align*}
          where we have used \eqref{elliptic_order} and \eqref{holder_k} for the last inequality. Combining these estimates we obtain
              \begin{align*}
                  &|L^\theta_{y_1} \phi(x) - L^\theta_{y_2} \phi(x)| \leq \int_{\RN} |\phi(x+z) - \phi(x)| |K^\theta(y_1, z) - K^\theta(y_2, z)| \,dz\\
                  &\quad \leq C|y_1-y_2|^\alpha \int_{\RN} |\phi(x+z) - \phi(x)| |z|^{N+\sigma}\,dz\\
                  &\quad \leq C|y_1-y_2|^\alpha \|\phi\|_{\sigma+\iota}.
              \end{align*}
          Here we have removed the principal value from each of the integrals defining $L^\theta_{y_1} \phi$ and $L^\theta_{y_2} \phi$ by using that $\phi\in C^{\sigma+\iota}$ and computing as in part \textit{(\ref{dp_Linfty})} of the lemma. The third inequality is similarly obtained.
          
          Again using the assumption \eqref{holder_k}, we apply the results of \cite{S} to equation \eqref{v_lambda_eq}, and with this we conclude.\\
  
          \textit{(\ref{dp_diff})} Let $w_i = w_i(\cdot; x_i, p_1, \phi_i)$ for $x_i, p_1\in \RN$, $\phi_i\in C^{\sigma+\iota}(\RN)$, $i=1, 2$. Using the assumptions on the structure of $H$, it is easy to see that $w_\pm:= w_1 \pm \lambda^{-1}\big[|x_1-x_2|(1 + |p_2| + \|\phi_1\|_{\sigma+ \iota}) + |p_1 - p_2| + \|\phi_1 - \phi_2\|\big]$ are respectively a super- and a subsolution of \eqref{dp} centered in $(x_2, p_2, \phi_2)$. The claim follows by comparison.\\

          \textit{(\ref{dp_van_disc_rate})} Adapting the arguments of \cite{camilli2009rates}, \cite{IMT}, we consider fixed $(x, p, \phi)$ and define
              \begin{equation*}
                  \Gamma_\lambda = \lambda \sup_y w_\lambda (y; x, p, \phi).
              \end{equation*}
          We first claim that 
             \begin{equation}\label{gamma_lam}
                 \Gamma_\lambda \geq -\bar{H}(x,p,\phi).
             \end{equation}
          Indeed, by noting that $w_\lambda$ satisfies
              \begin{equation*}
                  \Gamma_\lambda + \sup_\theta\left\{ -L^\theta_y w_\lambda(y) - L^\theta_y \phi(x) - f^\theta(x,y)\cdot p - l^{\theta}(x,y) \right\} \geq 0,
              \end{equation*}
          the claim follows from applying the following characterization of the ergodic constant: using the notation of Section \ref{sec_eff},
              \begin{equation*}
                  c = \sup \{ \tilde c \in \R : \exists \ \psi \ \mbox{s.t.} \ F(\psi, y) \geq \tilde c \ \mbox{in} \ \mathbb{T}^N \}.
              \end{equation*}
          This is in a sense dual to \eqref{caract1} and can be proved in exactly the same way.
         
          We then note that part \textit{(\ref{dp_sig+al})} of the lemma implies in particular that
              \begin{equation*}
                  \lambda|w_\lambda(y_1; x, p, \phi) - w_\lambda(y_2; x, p, \phi)| \leq \lambda C_1 \sqrt{N} (1 + |p| + \|\phi\|_{\sigma +\iota}),
              \end{equation*}
          given that $\sqrt{N}$ is the diameter of the unit hypercube in $\RN$. Plugging \eqref{gamma_lam} into the previous inequality, we find
              \begin{equation*}
                  \lambda w_\lambda (y; x, p, \phi) \geq -\bar{H}(x,p,\phi) - \lambda C_1 \sqrt{N} (1 + |p| + \|\phi\|_{\sigma +\iota}) \quad \textrm{for all } y\in \mathbb{T}^N.
              \end{equation*}
          The proof of the corresponding upper bound is similarly obtained.
      \end{proof}
      
      \medskip
      
      In the proof of Theorem \ref{teoconv} we will make use of the following estimate.
          \begin{proposition}\label{prop_nonlocalDiff}
              Let $x_1, x_2, z\in \RN$, $y\in \mathbb{T}^N$, $\phi\in C^{\sigma+\alpha}(\RN)$. For any $\alpha'<\alpha$, there exists a $C>0$ such that
                  \begin{equation}\label{princNonlocal_diff}
                      |L^\theta_y \phi(x_1) - L^\theta_y \phi(x_2)| \leq C|x_1-x_2|^{\alpha'}\|\phi\|_{\sigma+\alpha}.
                  \end{equation}
          \end{proposition}
      
          \begin{proof}
              Write
                  \begin{equation*}
                      \delta := |\phi(x_1 + z) - \phi(x_1) - (\phi(x_2 + z) - \phi(x_2))|
                  \end{equation*}
              As (in particular) $\phi\in C^1$, we have
                  \begin{align*}
                      \delta={}& \left|\int_{0}^{1} \left(D\phi(x_1 + tz) - D\phi(x_2 + tz)\right) \cdot z \,dt \right|\\
                      \leq{}& \int_{0}^{1} |D\phi(x_1 + tz) - D\phi(x_2 + tz)| |z| \,dt \leq \int_{0}^{1} [\phi]_{1,\beta} |x_1-x_2|^\beta |z| \,dt\\
                      \leq{}& [\phi]_{1,\beta}\, |x_1-x_2|^\beta |z|.
                  \end{align*}
              Also, by rearranging the terms,
                  \begin{align*}
                      \delta={}& |\phi(x_1 + z) - \phi(x_2 + z) - (\phi(x_1)  - \phi(x_2))|\\
                      \leq{}& \int_{0}^{1} |D\phi(x_2 + z + t(x_1 - x_2)) - D\phi(x_2 + t(x_1 - x_2))| |x_1-x_2| \,dt\\
                      \leq{}& [\phi]_{1,\beta}\,|z|^\beta  |x_1-x_2|.
                  \end{align*}
              Hence
                  \begin{equation*}
                      \delta = \delta^s\delta^{1-s} \leq [\phi]_{1,\beta}\, |x_1-x_2|^{s\beta + 1-s} |z|^{s + \beta(1-s)}.
                  \end{equation*}
              Setting $s = \frac{1-\alpha}{1-\beta} + \nu$ for some $\nu>0$, we have
                  \begin{align*}
                      &|L^\theta_y \phi(x_1) - L^\theta_y \phi(x_2)|\\
                      &\quad\leq \int_{\RN} |\phi(x_1 + z) - \phi(x_1) - (\phi(x_2 + z) - \phi(x_2))| K^\theta(y,z) \,dz\\
                      &\quad\leq  [\phi]_{1,\beta}\, |x_1-x_2|^{\alpha-\nu(1-\beta)} \int_{\RN} |z|^{\sigma + \nu(1-\beta)} K^\theta(y,z)\, dz\\
                      &\quad\leq C [\phi]_{1,\beta}|x_1-x_2|^{\alpha-\nu(1-\beta)}.
                  \end{align*}
              We thus obtain \eqref{princNonlocal_diff} taking a sufficiently small $\nu$, noting $C$ depends on $\alpha':=\alpha-\nu(1-\beta)$ and universal constants.
          \end{proof}
  
  \section{Rate of convergence: Proof of Theorem~\ref{teoconv}.}\label{secrate}
     	
  We now provide the proof of Theorem~\ref{teoconv}.
  
     	\begin{proof}
     		The proof follows that of Theorem 2.1 in \cite{camilli2009rates}. Crucially, we will employ the following bound, consequence of Corollary \ref{eff_reg}:
              \begin{equation}\label{nonlocBounds}
                  \|u\|_{C^{\sigma+\alpha}(\RN)} \leq M.
              \end{equation}
          Let $M_0, b>0$ to be chosen, and $\psi\in C^\infty(\RN)$ a radial function, nondecreasing with respect to $|x|$, such that
     			\begin{equation*}
     				\psi(x) = \left\{
     					\begin{array}{ll}
     						|x|^2, &\text{if }|x|\leq 1,\\
     						M_0, & \text{if }|x|\geq 2.
     					\end{array}
     				\right.
     			\end{equation*}
          We immediately note that 
              \begin{equation}\label{penalBounds}
                  |D\psi|,\ |L^\theta_{\nicefrac{\tilde{x}}{\epsilon}} \psi|\leq M_1
              \end{equation}
          where $M_1$ is independent of $\epsilon$.
              
     		Writing $w_\lambda (\cdot; [u](x)) = w_\lambda(\cdot; x, Du(x), u)$ for short, we define
     			\begin{equation*}
     				\varphi(x) = u^\epsilon(x) - u(x) - \epsilon^\sigma w_\lambda(\nicefrac{x}{\epsilon}; [u](x)) - b\psi(x),
     			\end{equation*}
  	   and note that for sufficiently large $M_0$, $\varphi$ attains a \textit{global} maximum at some $\hat{x}\in B_2$.
     		
     		For $c>0$, define 
     			\begin{equation*}
     				\tilde{\varphi}(x) = u^\epsilon(x) - u(x) - \epsilon^\sigma w_\lambda(\nicefrac{x}{\epsilon}; [u](\hat{x})) - b\psi(x) - c\psi(x-\hat{x}).
     			\end{equation*}
     		Given $\tau>0$, we claim that $c$ can be chosen so that $\tilde{\varphi}$ attains a global maximum $\tilde{x}\in B_\tau(\hat{x})$. Indeed, set		
     			\begin{equation*}
     				c=\frac{\epsilon^\sigma}{\lambda\tau^{2-\alpha'}}2C (1 + 2M)
     			\end{equation*}
     		for $C>C_1$, where $C_1$ and $\alpha'$ are as in Lemma \ref{lemma_dp_bounds} and Proposition \ref{prop_nonlocalDiff}, respectively, and let $x\in \RN\backslash B_\tau(\hat{x})$. By construction we have $\tilde{\varphi}(\hat{x}) = \varphi(\hat{x}) \geq \varphi(x)$, hence
     			\begin{align*}
     				\tilde{\varphi}(\hat{x}) - \tilde{\varphi}(x) &{}\geq \varphi(x) - \tilde{\varphi}(x) = -\epsilon^\sigma\left[w_\lambda(\nicefrac{x}{\epsilon}; [u](x)) - w_\lambda(\nicefrac{x}{\epsilon}; [u](\hat{x})) \right] + c\psi(x-\hat{x}).
     			\end{align*}        	
     		Assume $\tau<1$ (it will eventually be taken small). If $\tau\leq |x-\hat{x}| < 1$ we have 
     			\begin{equation*}
     				c\psi(x-\hat{x}) = c|x-\hat{x}|^2 = \frac{\epsilon^\sigma}{\lambda\tau^{2-\alpha'}}2C (1 + 2M)|x-\hat{x}|^2> \frac{\epsilon^\sigma}{\lambda}2C_1 (1 + 2M)|x-\hat{x}|^{\alpha'},
     			\end{equation*}
     		while Lemma \ref{lemma_dp_bounds}\eqref{dp_diff} and \eqref{princNonlocal_diff} gives that
     			\begin{align*}
     				\left[ w_\lambda(\nicefrac{x}{\epsilon}; [u](x)) - w_\lambda(\nicefrac{x}{\epsilon}; [u](\hat{x})) \right] \leq \lambda^{-1}C_1 (1 + 2M)|x-\hat{x}|^{\alpha'}.
     			\end{align*}
     		Combining these inequalities we obtain that $\tilde{\varphi}(\hat{x}) - \tilde{\varphi}(x)>0$. If, on the contrary, $|x-\hat{x}|\geq 1$, then $c\psi(x-\hat{x})\geq c$, and from Lemma \ref{lemma_dp_bounds}\eqref{dp_Linfty}, again using $\tau<1$, we have
     			\begin{align*}
     				\epsilon^\sigma \left[ w_\lambda(\nicefrac{x}{\epsilon}; [u](x)) - w_\lambda(\nicefrac{x}{\epsilon}; [u](\hat{x})) \right] \leq \epsilon^\sigma 2\lambda^{-1}C_1 (1 + 2M) < c.
     			\end{align*}
     		This again implies that $\tilde{\varphi}(\hat{x}) - \tilde{\varphi}(x)>0$. Thus $\tilde{\varphi}$ attains a global maximum in $B_\tau(\hat{x})$, which we denote by $\tilde{x}$.
          
          In particular, we have the following behavior at $\tilde{x}$:
              \begin{align}\label{maxXtilde1}
                  Du^\epsilon(\tilde{x}) = Du(\tilde{x}) + \epsilon^{\sigma-1}D_y w_\lambda(\nicefrac{\tilde{x}}{\epsilon}; [u](\hat{x})) - bD\psi(\tilde{x}) - 2c(\tilde{x} - \hat{x}),
              \end{align}
          since $|\tilde{x} - \hat{x}| < \tau < 1$ implies that $\psi(x - \hat{x}) = c|x - \hat{x}|^2$ in a neighborhood of $\tilde{x}$, and
              \begin{align}\label{maxXtxtildeNonloc}
                  &L^\theta_{\nicefrac{\tilde{x}}{\epsilon}} u^\epsilon(\tilde{x}) \nonumber\\
                  &\quad \leq L^\theta_{\nicefrac{\tilde{x}}{\epsilon}} u(\tilde{x}) + \epsilon^\sigma L^\theta_{\nicefrac{\tilde{x}}{\epsilon}} w_\lambda(\nicefrac{\cdot}{\epsilon}; [u](\hat{x}))(\tilde{x}) + bL^\theta_{\nicefrac{\tilde{x}}{\epsilon}} \psi(\tilde{x}) + cL^\theta_{\nicefrac{\tilde{x}}{\epsilon}} \psi(\cdot - \hat{x})(\tilde{x})\nonumber\\
                  &\quad = L^\theta_{\nicefrac{\tilde{x}}{\epsilon}} u(\tilde{x}) + L^\theta_{\nicefrac{\tilde{x}}{\epsilon}}, w_\lambda(\cdot; [u](\hat{x}))(\nicefrac{\tilde{x}}{\epsilon}) + bL^\theta_{\nicefrac{\tilde{x}}{\epsilon}} \psi(\tilde{x}) + cL^\theta_{\nicefrac{\tilde{x}}{\epsilon}} \psi(\tilde{x}-\hat{x}),
              \end{align}
          using the homogeneity and translation invariance of $L^\theta$.
          
          For the following computation we write $w_\lambda = w_\lambda(\cdot; [u](\hat{x}))$ to ease notation. Evaluating \eqref{eq} at $\tilde{x}$, we have
              \begin{equation*}
                  u^\epsilon(\tilde{x}) + \sup_{\theta} \left\{- L^\theta_{\nicefrac{\tilde{x}}{\epsilon}} u^\epsilon(\tilde{x}) - f^\theta(\tilde{x}, \nicefrac{\tilde{x}}{\epsilon})\cdot Du^\epsilon(\tilde{x}) - l^{\theta}(\tilde{x}, \nicefrac{\tilde{x}}{\epsilon}) \right\} = 0,
              \end{equation*}
          and continue estimating:
              \begin{align*}
                  &\sup_{\theta} \left\{- L^\theta_{\nicefrac{\tilde{x}}{\epsilon}} u^\epsilon(\tilde{x}) - f^\theta(\tilde{x}, \nicefrac{\tilde{x}}{\epsilon})\cdot Du^\epsilon(\tilde{x}) - l^{\theta}(\tilde{x}, \nicefrac{\tilde{x}}{\epsilon}) \right\}\\      
                  &\quad\geq \sup_{\theta} \left\{ - L^\theta_{\nicefrac{\tilde{x}}{\epsilon}} u(\tilde{x}) - L^\theta_{\nicefrac{\tilde{x}}{\epsilon}} w_\lambda(\nicefrac{\tilde{x}}{\epsilon}) - bL^\theta_{\nicefrac{\tilde{x}}{\epsilon}} \psi(\tilde{x}) - cL^\theta_{\nicefrac{\tilde{x}}{\epsilon}} \psi(\tilde{x}-\hat{x}) \right.\\
                  &\qquad \left. - f^\theta(\tilde{x}, \nicefrac{\tilde{x}}{\epsilon})\cdot \big(Du(\tilde{x}) + \epsilon^{\sigma-1}D_y w_\lambda(\nicefrac{\tilde{x}}{\epsilon}) - bD\psi(\tilde{x}) - 2c(\tilde{x} - \hat{x})\big) - l^{\theta}(\tilde{x}, \nicefrac{\tilde{x}}{\epsilon}) \right\}\\
                  &\quad \geq \sup_{\theta} \left\{ - L^\theta_{\nicefrac{\tilde{x}}{\epsilon}} u(\tilde{x}) - L^\theta_{\nicefrac{\tilde{x}}{\epsilon}} w_\lambda(\nicefrac{\tilde{x}}{\epsilon}) - f^\theta(\tilde{x}, \nicefrac{\tilde{x}}{\epsilon})\cdot Du(\tilde{x}) - l^{\theta}(\tilde{x}, \nicefrac{\tilde{x}}{\epsilon})\right\}\\
                  &\qquad -M_2\big(\epsilon^{\sigma-1} + b + (2\tau +1)c\big)\\
                  &\quad \geq \sup_{\theta} \left\{ - L^\theta_{\nicefrac{\tilde{x}}{\epsilon}} u(\hat{x}) - L^\theta_{\nicefrac{\tilde{x}}{\epsilon}} w_\lambda(\nicefrac{\tilde{x}}{\epsilon}) - f^\theta(\hat{x}, \nicefrac{\tilde{x}}{\epsilon})\cdot Du(\hat{x}) - l^{\theta}(\hat{x}, \nicefrac{\tilde{x}}{\epsilon})\right\}\\
                  &\qquad -M_2\big(\tau^{\alpha'} + \tau  \big) -M_2\big(\epsilon^{\sigma-1} + b + (2\tau +1)c\big).
              \end{align*}
          Here we have used \eqref{maxXtilde1} and \eqref{maxXtxtildeNonloc} for the first inequality, the bounds \eqref{nonlocBounds} and \eqref{penalBounds} from the second (and onward) together with Part (b), estimate \eqref{princNonlocal_diff} for the third, and chosen a sufficiently large $M_2$---in particular, such that $M_2\geq M_0, M_1$, and independent of $\epsilon, \lambda, \tau, b$ and $c$.
          
          We now use that $w_\lambda(\cdot; [u](\hat{x}))$ and $u$ are solutions of \eqref{dp} and \eqref{eqnonloceff}, respectively, and Lemma \ref{lemma_dp_bounds} \eqref{dp_van_disc_rate} to obtain
              \begin{align*}
                  &\sup_{\theta} \left\{ - L^\theta_{\nicefrac{\tilde{x}}{\epsilon}}, u(\hat{x}) - L^\theta_{\nicefrac{\tilde{x}}{\epsilon}} w_\lambda(\nicefrac{\tilde{x}}{\epsilon}) - f^\theta(\hat{x}, \nicefrac{\tilde{x}}{\epsilon})\cdot Du(\hat{x}) - l^{\theta}(\hat{x}, \nicefrac{\tilde{x}}{\epsilon})\right\}\\
                  &\quad = -\lambda w_\lambda(\nicefrac{\tilde{x}}{\epsilon}; [u](\hat{x})) \geq \overline{H}(\hat{x}, Du(\hat{x}), u) - \lambda C_1(1 + 2M)\\
                  &\quad \geq -u(\hat{x}) - \lambda C_1(1 + 2M).
              \end{align*}
          Combining this with the previous computation we have
              \begin{equation*}
                  u^\epsilon(\tilde{x}) - u(\hat{x}) \leq M_2 \big(\lambda + \epsilon^{\sigma-1} + b + c + \tau^{\alpha'}\big).
              \end{equation*}
          By construction, for all $x\in\RN$ we have $\varphi(x) \leq \varphi(\hat{x}) = \tilde{\varphi}(\hat{x}) \leq \tilde{\varphi}(\tilde{x})$, hence
              \begin{align*}
                  u^\epsilon(x) - u(x) \leq{}& \big(u^\epsilon(\tilde{x}) - u(\hat{x})\big) + \big(u(\hat{x}) - u(\tilde{x})\big)\\
                  &\quad + \epsilon^\sigma\left[w_\lambda\left(\nicefrac{x}{\epsilon}; [u](x)\right) - w_\lambda\left(\nicefrac{\tilde{x}}{\epsilon}; [u](\hat{x})\right)\right] + b[\psi(x) - \psi(\tilde{x})]\\
                  &\quad \leq M_2 \big(\lambda + \epsilon^{\sigma-1} + b + c + \tau^{\alpha'}\big) + M\tau + \frac{\epsilon^\sigma}{\lambda}2C_1(1 + 2M) + M_0 b,
              \end{align*}
          by Lemma \ref{lemma_dp_bounds} \eqref{dp_Linfty}, \eqref{penalBounds} and \eqref{nonlocBounds}.
          
          Sending $b\to 0$ and recalling the definition of $c$, this gives
              \begin{equation*}
                  u^\epsilon(x) - u(x)  \leq M_2 \left(\lambda + \epsilon^{\sigma-1} + \frac{\epsilon^\sigma}{\lambda\tau^{2-\alpha'}} + \tau^{\alpha'}\right),
              \end{equation*}
          again by taking a larger value of $M_2$ if necessary. Thus, by symmetry, we set
              \begin{equation*}
                  \lambda = \epsilon^{\frac{\sigma\alpha'}{2 + \alpha'}}, \quad \tau = \epsilon^{\frac{\sigma}{2+\alpha'}},
              \end{equation*}
          and obtain
              \begin{equation*}
                   u^\epsilon(x) - u(x) \leq 4M_2\epsilon^{\frac{\sigma\alpha'}{2 + \alpha'}}.
              \end{equation*}
      \end{proof}

      \begin{remark}\label{rmk_aPosteriori}
          The proof on Theorem \ref{teoconv} does not rely on establishing the convergence $u^\epsilon\to u$ as $\epsilon\to 0$ beforehand. Consequently, Theorem \ref{teoconv} trivially implies the convergence result (i.e., homogenization).
      \end{remark}
      
      \medskip
      
      We proceed with the proof of the easier results on rates of convergence when some of the dependencies of $H$ are dropped. We note that these results are independent of Theorems \ref{corbarH} and \ref{teoconv}.
      
      \begin{proof}[Proof of Theorem \ref{teo_simple_rate}]
          \textit{(i)} From the assumptions, \eqref{eq} may be written as 
              \begin{equation}\label{easy_eq1}
                  u^\epsilon + H(x, y, u^\epsilon) = 0 \quad \textrm{in }\RN, 
              \end{equation}
          where 
              \begin{equation*}
                  H(x,y,\varphi) = \sup_{\theta \in \Theta} \{-L^\theta_y \varphi(x) -l^{\theta}(y)\}.
              \end{equation*}
          As a function in $\RN\times\mathbb{T}^N\times C^2_b(\RN)$, $H$ in fact independent of $x$, since $L^\theta$ depends on $x$ only through evaluating $u(x+z) - u(x)$ in the integrand of $L^\theta$ and $K^\theta$ is translation invariant. It is however not entirely correct to write $H(y, u^\epsilon)$ instead of $H(x,y, u^\epsilon)$ above. The same can be said for the effective problem, 
              \begin{equation*}
                  u + \bar{H}(x,u) = 0 \quad\textrm{in }\RN, 
              \end{equation*}
          which has the constant solution $u\equiv -\bar{H}(0,0)$.
          
          It follows from the preceding remarks that the associated cell problem is
              \begin{equation}\label{easy_cp_1}
                  \sup_{\theta \in \Theta} \{ -L^\theta_y w(y) - l^{\theta}(y)\} = \bar{H}(0,0) \quad\textrm{in } \mathbb{T}^N,
              \end{equation}
          and has a unique $\mathbb{T}^N$-periodic solution $w=w(y)$, by the results of Lemma \ref{lemma_dp_bounds}.
          
          Define $\bar{v}(x) = u(x) + \epsilon^\sigma w(\nicefrac{x}{\epsilon}) + \epsilon^\sigma\|w\|_\infty = -\bar{H}(0,0) + \epsilon^\sigma w(\nicefrac{x}{\epsilon}) + \epsilon^\sigma\|w\|_\infty$. Substituting in \eqref{eq} and using \eqref{easy_cp_1}, we have
              \begin{align*}
                  \bar{v}(x) + H(x,\nicefrac{x}{\epsilon}, \bar{v}) ={}& -\bar{H}(0,0) + w(\nicefrac{x}{\epsilon}) + \epsilon^\sigma\|w\|_\infty + H(x,\nicefrac{x}{\epsilon}, \epsilon^\sigma w(\nicefrac{\cdot}{\epsilon}))\\
                  =&{} -\bar{H}(0,0) + w(\nicefrac{x}{\epsilon}) + \epsilon^\sigma\|w\|_\infty + H(\nicefrac{x}{\epsilon},\nicefrac{x}{\epsilon}, w) \geq 0.
              \end{align*}
          Thus by comparison, $u^\epsilon \leq \bar{v}$, i.e., $u^\epsilon - u \leq \epsilon^\sigma w(\nicefrac{x}{\epsilon}) + \epsilon^\sigma\|w\|_\infty \leq 2 \epsilon^\sigma\|w\|_\infty.$ The lower bound is similarly obtained.
          
          \textit{(ii)} Arguing as in the first part of the proof we have that $u^\epsilon$ and $u\equiv -H(0,0,0)$ are respectively solutions of
              \begin{equation}\label{easy_eq2}
                  u^\epsilon + H(x, \nicefrac{x}{\epsilon}, Du^\epsilon, u^\epsilon) = 0, \quad \textrm{in }\RN, 
              \end{equation}
          where $H(x, y, p, \varphi) =  \sup_{\theta \in \Theta} \{-L^\theta_y \varphi(x) - f^\theta(y)\cdot p - l^{\theta}(y)\},$ and
              \begin{equation*}
                  u + H(x, Du, u) = 0 \quad \textrm{in }\RN. 
              \end{equation*}
          The associated cell problem is
              \begin{equation}\label{easy_cp_2}
                  H(x,y, 0, w) = \bar{H}(0,0) \quad\textrm{in } \mathbb{T}^N,
              \end{equation}
          and we define $\bar{v}(x) = -H(0,0,0) + \epsilon^\sigma w(\nicefrac{x}{\epsilon}) + C\epsilon^{\sigma-1}\|w\|_{C^1}$, where $C$ is given by \eqref{lip_data}. Again, evaluating in \eqref{easy_eq2}, we have
              \begin{align*}
                  &\bar{v}(x) + H(x,\nicefrac{x}{\epsilon}, D\bar{v}(x), \bar{v})\\
                  &\quad = -H(0,0,0) + \epsilon^\sigma w(\nicefrac{x}{\epsilon}) + C\epsilon^{\sigma-1}\|w\|_{C^1} + H(x, \nicefrac{x}{\epsilon}, \epsilon^{\sigma-1} Dw(\nicefrac{x}{\epsilon}), \epsilon^\sigma w(\nicefrac{\cdot}{\epsilon}))\\
                  &\quad = -H(0,0,0) + \epsilon^\sigma w(\nicefrac{x}{\epsilon}) + C\epsilon^{\sigma-1}\|w\|_{C^1} + H(x, \nicefrac{x}{\epsilon},  0, \epsilon^\sigma w(\nicefrac{\cdot}{\epsilon}))\\
                  &\qquad  - \epsilon^{\sigma-1}\sup_{\theta \in \Theta} \{f^\theta(\nicefrac{x}{\epsilon})\cdot Dw(\nicefrac{x}{\epsilon})\}\\
                  &\quad \geq 0.
              \end{align*}
          We conclude once more that $u^\epsilon \leq \bar{v}$, hence $u^\epsilon - u \leq 2C\epsilon^{\sigma-1} \|w\|_{C^1}$.
      \end{proof}
  
 \noindent {\bf Acknowledgements.} 
  
 A.~R.-P.~was partially supported by Fondecyt Grant Postdoctorado Nacional 2019 No.~3190858. E.~T.~was partially supported by Fondecyt No.~1201897.

 \end{document}